\documentclass[11pt]{amsart}
\usepackage{amsmath}
\usepackage{amssymb}
\usepackage{amsthm}
\usepackage{amscd}
\usepackage{enumerate}
\usepackage{verbatim}
\usepackage[all]{xy}
\usepackage{mathrsfs}

\theoremstyle{plain}
\newtheorem{thm}{Theorem}[section]
\newtheorem{prop}[thm]{Proposition}
\newtheorem{lem}[thm]{Lemma}

\theoremstyle{definition}

\newtheorem{rmk}[thm]{Remark}

\newcommand{\Hom}{\mathrm{Hom}}

\newcommand{\gr}{\mathrm{gr}}
\newcommand{\Ker}{\mathrm{Ker}}

\newcommand{\prjt}{\mathrm{pr}}
\newcommand{\Fil}{\mathrm{Fil}}
\newcommand{\Spf}{\mathrm{Spf}}
\newcommand{\Spv}{\mathrm{Sp}}
\newcommand{\Spec}{\mathrm{Spec}}

\newcommand{\Frac}{\mathrm{Frac}}

\newcommand{\Hdg}{\mathrm{Hdg}}

\newcommand{\et}{\mathrm{et}}

\newcommand{\sep}{\mathrm{sep}}

\newcommand{\Kbar}{\bar{K}}

\newcommand{\okbar}{\mathcal{O}_{\bar{K}}}

\newcommand{\Acrys}{A_{\mathrm{crys}}}

\newcommand{\Bcrys}{B_{\mathrm{crys}}}

\newcommand{\okey}{\mathcal{O}_K}

\newcommand{\oel}{\mathcal{O}_L}

\newcommand{\cC}{\mathcal{C}}
\newcommand{\cD}{\mathcal{D}}

\newcommand{\cG}{\mathcal{G}}
\newcommand{\cH}{\mathcal{H}}

\newcommand{\cK}{\mathcal{K}}

\newcommand{\cM}{\mathcal{M}}
\newcommand{\cN}{\mathcal{N}}
\newcommand{\cO}{\mathcal{O}}

\newcommand{\Tcrys}{T_{\mathrm{crys}}^{*}}

\newcommand{\TSG}{T^{*}_{\SG}}

\newcommand{\Vcrys}{V^{*}_{\mathrm{crys}}}

\newcommand{\okbari}{\mathcal{O}_{\Kbar,i}}

\newcommand{\PD}{\mathrm{DP}}

\newcommand{\CRYS}{\mathrm{CRYS}}

\newcommand{\SG}{\mathfrak{S}}

\newcommand{\SGm}{\mathfrak{M}}
\newcommand{\SGn}{\mathfrak{N}}
\newcommand{\SGl}{\mathfrak{L}}

\newcommand{\ModSGf}{\mathrm{Mod}_{/\mathfrak{S}_1}^{1,\varphi}}
\newcommand{\ModSGfinf}{\mathrm{Mod}_{/\mathfrak{S}_{\infty}}^{1,\varphi}}
\newcommand{\ModSGffr}{\mathrm{Mod}_{/\mathfrak{S}}^{1,\varphi}}

\newcommand{\ModSf}{\mathrm{Mod}_{/S_1}^{1,\varphi}}
\newcommand{\ModSfinf}{\mathrm{Mod}_{/S_\infty}^{1,\varphi}}
\newcommand{\ModSffr}{\mathrm{Mod}_{/S}^{1,\varphi}}

\newcommand{\upi}{\underline{\pi}}

\newcommand{\um}{\underline{m}}
\newcommand{\un}{\underline{n}}
\newcommand{\up}{\underline{p}}
\newcommand{\uY}{\underline{Y}}

\newcommand{\bD}{\mathbb{D}}

\newcommand{\bZ}{\mathbb{Z}}
\newcommand{\bQ}{\mathbb{Q}}

\newcommand{\sI}{\mathscr{I}}
\newcommand{\sJ}{\mathscr{J}}
\newcommand{\sS}{\mathscr{S}}

\newcommand{\frB}{\mathfrak{B}}
\newcommand{\frG}{\mathfrak{G}}

\newcommand{\frX}{\mathfrak{X}}

\makeatletter
\renewcommand{\p@enumii}{}
\makeatother

\begin{document}

\title[On lower ramification subgroups]{On lower ramification subgroups and canonical subgroups}
\author{Shin Hattori}
\date{\today}
\email{shin-h@math.kyushu-u.ac.jp}
\address{Faculty of Mathematics, Kyushu University}
\thanks{Supported by JSPS KAKENHI Grant Number 23740025.}

\begin{abstract}
Let $p$ be a rational prime, $k$ be a perfect field of characteristic $p$ and $K$ be a finite totally ramified extension of the fraction field of the Witt ring of $k$. Let $\cG$ be a finite flat commutative group scheme over $\okey$ killed by some $p$-power. In this paper, we prove a description of ramification subgroups of $\cG$ via the Breuil-Kisin classification, generalizing the author's previous result on the case where $\cG$ is killed by $p\geq 3$. As an application, we also prove that the higher canonical subgroup of a level $n$ truncated Barsotti-Tate group $\cG$ over $\okey$ coincides with lower ramification subgroups of $\cG$ if the Hodge height of $\cG$ is less than $(p-1)/p^n$. 
\end{abstract}

\maketitle

\section{Introduction}\label{intro}

Let $p$ be a rational prime, $k$ be a perfect field of characteristic $p$ and $W=W(k)$ be the Witt ring of $k$. The natural Frobenius endomorphism of the ring $W$ lifting the $p$-th power Frobenius of $k$ is denoted by $\varphi$. Let $K$ be a finite extension of $K_0=\Frac(W)$ with integer ring $\okey$, uniformizer $\pi$ and absolute ramification index $e$. We fix an algebraic closure $\Kbar$ of $K$ and extend the valuation $v_p$ of $K$ satisfying $v_p(p)=1$ to $\Kbar$. Let $\hat{\cO}_{\Kbar}$ be the completion of the integer ring $\cO_{\Kbar}$. We also fix a system $\{\pi_n\}_{n\geq 0}$ of $p$-power roots of $\pi$ in $\Kbar$ satisfying $\pi_0=\pi$ and $\pi_{n+1}^p=\pi_n$ and put $K_\infty=\cup_n K(\pi_n)$. The absolute Galois groups of $K$ and $K_\infty$ are denoted by $G_K$ and $G_{K_\infty}$, respectively. For any positive rational number $i$, put
$m_K^{\geq i}=\{x\in \okey| v_p(x)\geq i\}$ and $\cO_{K,i}=\okey/m_K^{\geq i}$. For any valuation ring $V$ of height one, we define $m_V^{\geq i}$ and $V_i$ similarly. We also put $\sS_i=\Spec(\cO_{K,i})$, $\sS_{L,i}=\Spec(\cO_{L,i})$ for any finite extension $L/K$ and $\bar{\sS}_i=\Spec(\cO_{\Kbar,i})$.

Breuil conjectured a classification of finite flat (commutative) group schemes over $\okey$ killed by some $p$-power via $\varphi$-modules over the formal power series ring $\SG=W[[u]]$ and obtained such a classification for the case where groups are killed by $p\geq 3$ (\cite{Br_AZ}). It is often referred to as the Breuil-Kisin classification, since Kisin showed the conjecture for $p\geq 3$ (\cite{Ki_Fcrys}) and for the case where $p=2$ and groups are connected (\cite{Ki_2}). The conjecture was proved for any $p$ independently by Kim (\cite{Kim_2}), Lau (\cite{Lau_2}) and Liu (\cite{Li_2}). In particular, we have an exact category $\ModSGf$ of such $\varphi$-modules over $\SG_1=\SG/p\SG$ (for the definition, see Section \ref{BK}) and an anti-equivalence of exact categories $\SGm^*(-)$ from the category of finite flat group schemes over $\okey$ killed by $p$ to the category $\ModSGf$.

On the other hand, we also have an anti-equivalence $\cH(-)$ from the category $\ModSGf$ to an exact category of finite flat generically etale group schemes over $k[[u]]$ whose Verschiebung is zero (\cite[Th\'eor\`eme 7.4]{SGA3-7A}). Let $R$ be the valuation ring defined as the projective limit of $p$-th power maps
\[
R=\varprojlim (\cO_{\Kbar,1}\gets\cO_{\Kbar,1}\gets\cdots),
\]
which is considered as an $\SG$-algebra by the map $u\mapsto \upi=(\pi_0,\pi_1,\ldots)$. The ring $R$ admits a natural $G_K$-action. We normalize its valuation $v_R$ by $v_R(\upi)=1/e$. We also normalize the indices of the upper and the lower ramification subgroups of finite flat generically etale group schemes $\cG$ over $\okey$ and $\cH$ over $k[[u]]$ to be adapted to $v_p$ and $v_R$, respectively. In particular, we define the $i$-th lower ramification subgroups of $\cG$ and $\cH$ by
\[
\cG_i(\okbar)=\Ker(\cG(\okbar)\to \cG(\cO_{\Kbar,i})), \quad \cH_i(R)=\Ker(\cH(R)\to \cH(R_i)).
\]

By the classification theory mentioned above, we have an isomorphism of $G_{K_\infty}$-modules
\[
\varepsilon_\cG: \cG(\okbar)\to \cH(\SGm^*(\cG))(R).
\]
In this paper, we prove that this isomorphism respects the upper and the lower ramification subgroups of both sides, generalizing the main theorem of \cite{Ha_ramcorr} which showed the case of $p\geq 3$. 

\begin{thm}\label{ramcorr2}
Let $p$ be a rational prime and $K/\bQ_p$ be an extension of complete discrete valuation fields with perfect residue field $k$. Let $\cG$ be a finite flat group scheme over $\okey$ killed by $p$ and consider the associated object $\SGm^*(\cG)$ of the category $\ModSGf$. Then the map $\varepsilon_\cG: \cG(\okbar)\simeq \cH(\SGm^*(\cG))(R)$ induces the isomorphisms of $G_{K_\infty}$-modules
\[
\cG_i(\okbar)\simeq \cH(\SGm^*(\cG))_i(R),\quad \cG^j(\okbar)\simeq \cH(\SGm^*(\cG))^j(R)
\]
for any positive rational numbers $i$ and $j$.
\end{thm}

In fact, we prove a more general result which is valid for any $\cG$ killed by some $p$-power (Theorem \ref{lowramcorrn}). As an application of this general result, we also prove the following theorem.

\begin{thm}\label{canlow}
Let $K/\bQ_p$ be an extension of complete discrete valuation fields. Let $\cG$ be a truncated Barsotti-Tate group of level $n$, height $h$ and dimension $d$ over $\okey$ with $0<d<h$ and Hodge height $w<(p-1)/p^n$. Then the level $n$ canonical subgroup $\cC_n$ of $\cG$ (\cite[Theorem 1.1]{Ha_cansubZ}) satisfies the equalities $\cC_n=\cG_{i_n}=\cG_{i'_n}$ for 
\[
i_n=1/(p^{n-1}(p-1))-w/(p-1),\quad i'_n=1/(p^n(p-1)).
\]
\end{thm}

Note that by our assumption and \cite[Theorem 1.1]{Ha_cansubZ}, we have an isomorphism of groups $\cC_n(\okbar)\simeq (\bZ/p^n\bZ)^d$. The fact that the lower ramification subgroup $\cG_{i_n}(\okbar)$ is isomorphic to $(\bZ/p^n\bZ)^d$ for $w<(p-1)/p^n$ was proved by Rabinoff (\cite[Theorem 1.9]{Ra}) for the case where $K/\bQ_p$ is an extension of (not necessarily discrete) complete valuation fields of height one, by a different method. Theorem \ref{canlow} reproves this result of Rabinoff for the case where the base field $K$ is a complete discrete valuation field, and also shows that the canonical subgroup constructed by Rabinoff coincides with $\cC_n$. In particular, we show that it has standard properties as in \cite[Theorem 1.1]{Ha_cansubZ}, such as the coincidence with the kernel of a Frobenius homomorphism.

Using Theorem \ref{canlow}, we also prove the following theorem on a family construction of canonical subgroups, which is stronger than \cite[Corollary 1.2]{Ha_cansubZ}.

\begin{thm}\label{familycan}
Let $K/\bQ_p$ be an extension of complete discrete valuation fields. Let $\frX$ be an admissible formal scheme over $\Spf(\okey)$ and $\frG$ be a truncated Barsotti-Tate group of level $n$ over $\frX$ of constant height $h$ and dimension $d$ with $0<d<h$. We let $X$ and $G$ denote the Raynaud generic fibers of the formal schemes $\frX$ and $\frG$, respectively. Put $r_n=(p-1)/p^n$ and let $X(r_n)$ be the admissible open subset of $X$ defined by 
\[
X(r_n)(\Kbar)=\{x\in X(\Kbar)\mid \Hdg(\frG_x)<r_n\}.
\]
Then there exists an admissible open subgroup $C_n$ of $G|_{X(r_n)}$ over $X(r_n)$ such that, etale locally on $X(r_n)$, the rigid-analytic group $C_n$ is isomorphic to the constant group $(\bZ/p^n\bZ)^d$ and, for any finite extension $L/K$ and $x\in X(L)$, the fiber $(C_n)_x$ coincides with the generic fiber of the level $n$ canonical subgroup of $\frG_x$.
\end{thm}

In \cite{Ha_ramcorr}, the proof of Theorem \ref{ramcorr2} for $p\geq 3$ is reduced to showing a congruence of the defining equations of $\cG$ and $\cH(\SGm^*(\cG))$ with respect to the identification $k[[u]]/(u^e)\simeq \cO_{K,1}$ sending $u$ to $\pi$. This congruence is a consequence of an explicit description of the affine algebra of $\cG$ in terms of $\SGm^*(\cG)$ due to Breuil (\cite[Proposition 3.1.2]{Br}), which is known only for the case where $\cG$ is killed by $p\geq 3$. Here, instead, we study a relationship between the groups $\cG(\cO_{\Kbar,i})$ and $\cH(\SGm^*(\cG))(R_i)$ by using the faithfulness of the crystalline Dieudonn\'{e} functor (\cite{deJM}), from which Theorem \ref{ramcorr2} (and Theorem \ref{lowramcorrn}) follows easily.



\section{The Breuil-Kisin classification}\label{BK}

In this section, we briefly recall the classification of finite flat group schemes and Barsotti-Tate groups over $\okey$ due to Kim (\cite{Kim_2}). Consider the continuous $\varphi$-semilinear endomorphism of $\SG$ defined by $u\mapsto u^p$, which is denoted also by $\varphi$. Put $\SG_n=\SG/p^n\SG$. Let $E(u)\in W[u]$ be the (monic) Eisenstein polynomial of the uniformizer $\pi$. Then a Kisin module (of $E$-height $\leq 1$) is an $\SG$-module endowed with a $\varphi$-semilinear map $\varphi_\SGm:\SGm\to \SGm$, which we also write abusively as $\varphi$, such that the cokernel of the map 
\[
1\otimes \varphi:\varphi^*\SGm=\SG\otimes_{\varphi,\SG}\SGm\to \SGm
\]
is killed by $E(u)$. The Kisin modules form an exact category in an obvious manner, and its full subcategory consisting of $\SGm$ such that $\SGm$ is free of finite rank over $\SG$ ({\it resp.} free of finite rank over $\SG_1$ {\it resp.} finitely generated, $p$-power torsion and $u$-torsion free) is denoted by $\ModSGffr$ ({\it resp.} $\ModSGf$ {\it resp.} $\ModSGfinf$).

We also have categories of Breuil modules $\ModSffr$, $\ModSf$ and $\ModSfinf$ defined as follows (for more precise definitions, see for example \cite[Subsection 2.1]{Ha_ramcorr}, where the definitions are valid also for $p=2$). Let $S$ be the $p$-adic completion of the divided power envelope of $W[u]$ with respect to the ideal $(E(u))$ and put $S_n=S/p^n S$. The ring $S$ has a natural divided power ideal $\Fil^1S$, a continuous $\varphi$-semilinear endomorphism defined by $u\mapsto u^p$ which is also denoted by $\varphi$ and a differential operator $N:S\to S$ defined by $N(u)=-u$. We can also define a $\varphi$-semilinear map $\varphi_1=p^{-1}\varphi:\Fil^1S\to S$. Then a Breuil module (of Hodge-Tate weights in $[0,1]$) is an $S$-module endowed with an $S$-submodule $\Fil^1\cM$ containing $(\Fil^1S)\cM$ and a $\varphi$-semilinear map $\varphi_{1,\cM}:\Fil^1\cM\to \cM$ satisfying some conditions. We also define $\varphi_\cM: \cM\to \cM$ by $\varphi_\cM(x)=\varphi_1(E(u))^{-1}\varphi_{1,\cM}(E(u)x)$. We drop the subscript $\cM$ if there is no risk of confusion. The Breuil modules also form an exact category. Its full subcategory $\ModSffr$ ({\it resp.} $\ModSf$) is defined to be the one consisting of $\cM$ such that $\cM$ is free of finite rank over $S$ and $\cM/\Fil^1\cM$ is $p$-torsion free ({\it resp.} $\cM$ is free of finite rank over $S_1$). The category $\ModSfinf$ is defined as the smallest full subcategory containing $\ModSf$ and closed under extensions. Then the functor $\SGm\mapsto S\otimes_{\varphi,\SG}\SGm$ induces exact functors
\[
\ModSGffr \to \ModSffr,\quad \ModSGf \to \ModSf,\quad \ModSGfinf \to \ModSfinf
\]
which are all denoted by $\cM_\SG(-)$.

Put $\upi=(\pi_0,\pi_1,\ldots)\in R$ as before and consider the Witt ring $W(R)$ as an $\SG$-algebra by the map $u\mapsto [\upi]$. The $p$-adic period ring $\Acrys$ is defined as the $p$-adic completion of the divided power envelope of $W(R)$ with respect to the ideal $E(u)W(R)$ and the ring $\Acrys[1/p]$ is denoted by $\Bcrys^+$. For any $r=(r_0,r_1,\ldots)\in R$ with $r_l\in \cO_{\Kbar,1}$, choose a lift $\hat{r}_l$ of $r_l$ in $\okbar$ and put $r^{(m)}=\lim_{l\to\infty}\hat{r}_{l+m}^{p^l}\in \hat{\cO}_{\Kbar}$. Consider the surjection $\theta_n: W_n(R)\to \cO_{\Kbar,n}$ sending $(r_0,r_1,\ldots,r_{n-1})$ to $\sum_{l=0}^{n-1} p^l r_l^{(l)}$. Then the quotient $\Acrys/p^n\Acrys$ can be identified with the divided power envelope $W_n^\PD(R)$ of the surjection $\theta_n$ compatible with the canonical divided power structure on the ideal $pW_n(R)$. For any objects $\SGm\in \ModSGffr$ and $\cM\in \ModSffr$, we have the associated $G_{K_\infty}$-modules
\[
\TSG(\SGm)=\Hom_{\SG,\varphi}(\SGm,W(R)),\quad \Tcrys(\cM)=\Hom_{S,\varphi,\Fil^1}(\cM,\Acrys),
\]
which are related by the injection
\[
\TSG(\SGm)\to \Tcrys(\cM_\SG(\SGm))
\]
defined by $f\mapsto 1\otimes (\varphi\circ f)$. Similarly, for any object $\SGm\in \ModSGfinf$, we have the associated $G_{K_\infty}$-module
\[
\TSG(\SGm)=\Hom_{\SG,\varphi}(\SGm,\bQ_p/\bZ_p\otimes_{\bZ_p} W(R)).
\]

Let $D$ be an admissible filtered $\varphi$-module over $K$ such that $\gr^iD_K=0$ unless $i=0,1$. Put $S_{K_0}=S\otimes_W K_0$ and $\cD=S_{K_0}\otimes_{K_0} D$. The $S_{K_0}$-module $\cD$ is endowed with a natural Frobenius map $\varphi_\cD:\cD\to \cD$ induced by the Frobenius of $D$, a derivation $N_\cD=N\otimes 1: \cD\to \cD$ and an $S_{K_0}$-submodule $\Fil^1\cD$ defined as the inverse image of $\Fil^1D_K$ by the map $\cD\to \cD/(\Fil^1S) \cD=D_K$. Then a strongly divisible lattice in $\cD$ is an $S$-submodule $\cM$ of $\cD$ which satisfies the following:
\begin{itemize}
\item $\cM$ is a free $S$-module of finite rank and $\cD=\cM[1/p]$.
\item $\cM$ is stable under $\varphi_\cD$ and $N_\cD$.
\item $\varphi_\cD(\Fil^1\cM)\subseteq p\cM$, where $\Fil^1\cM=\cM\cap \Fil^1\cD$.
\end{itemize}
We put $\Vcrys(\cD)=\Hom_{S_{K_0},\varphi,\Fil^1}(\cD,\Bcrys^+)$. If $\cM$ is a strongly divisible lattice in $\cD$, then the natural $G_{K_\infty}$-actions on $\Tcrys(\cM)$ and $\Vcrys(\cD)=\Tcrys(\cM)[1/p]$ extend to $G_K$-actions and we have a natural isomorphism of $G_K$-modules
\[
\Vcrys(\cD)\to \Vcrys(D)=\Hom_{K_0,\varphi,\Fil}(D,\Bcrys^+)
\]
(\cite[Proposition 2.2.5]{Br_AZ} and \cite[Lemma 5.2.1]{Li_BC}).

Let $(\mathrm{BT}/\okey)$ ({\it resp.} $(p\text{-}\mathrm{Gr}/\okey)$) be the exact category of Barsotti-Tate groups ({\it resp.} finite flat group schemes killed by some $p$-power) over $\okey$. For any Barsotti-Tate group $\Gamma$ over $\okey$, we let $T_p(\Gamma)$ denote its $p$-adic Tate module, $V_p(\Gamma)=\bQ_p\otimes_{\bZ_p}T_p(\Gamma)$ and $D^*(\Gamma)$ be the filtered $\varphi$-module over $K$ associated to $V_p(\Gamma)$. We also let $\bD^*(-)$ denote the contravariant crystalline Dieudonn\'{e} functor (\cite{BBM}) and consider its module of sections 
\[
\bD^*(\Gamma)(S\to \okey)=\varprojlim_n \bD^*(\Gamma)(S_n\to \cO_{K,n})
\]
on the divided power thickening $S\to \okey$ defined by $u\mapsto \pi$. Note that the $S$-module $\bD^*(\Gamma)(S\to \okey)$ can be considered as an object of the category $\ModSffr$ and also as a strongly divisible lattice in $\cD^*(\Gamma)=S_{K_0}\otimes_{K_0}D^*(\Gamma)$ (\cite[Section 6]{Fal}). For any finite flat group scheme $\cG$ over $\okey$ killed by some $p$-power, we define an object $\bD^*(\cG)(S\to \okey)$ of the category $\ModSfinf$ similarly. Then we have the following classification theorem due to Kim, whose first assertion implies the second one by an argument of taking a resolution. 

\begin{thm}\label{KimBK}
\begin{enumerate}
\item\label{KimBK1} (\cite{Kim_2}, Theorem 4.1 and Proposition 4.2) There exists an anti-equivalence of exact categories
\[
\SGm^*(-):  (\mathrm{BT}/\okey) \to \ModSGffr
\]
with a natural isomorphism of $G_{K_\infty}$-modules 
\[
\varepsilon_\Gamma: T_p(\Gamma) \to \TSG(\SGm^*(\Gamma)).
\]
Moreover, the $S$-module $\cM_\SG(\SGm^*(\Gamma))$ can be considered as a strongly divisible lattice in $\cD^*(\Gamma)$ and we also have a natural isomorphism of strongly divisible lattices in $\cD^*(\Gamma)$
\[
\mu_\Gamma: \cM_\SG(\SGm^*(\Gamma))\to \bD^*(\Gamma)(S\to \okey).
\]

\item (\cite{Kim_2}, Corollary 4.3) There exists an anti-equivalence of exact categories
\[
\SGm^*(-):  (p\text{-}\mathrm{Gr}/\okey) \to \ModSGfinf
\]
with a natural isomorphism of $G_{K_\infty}$-modules 
\[
\varepsilon_\cG: \cG(\okbar) \to \TSG(\SGm^*(\cG)).
\]
Moreover, we also have a natural isomorphism of the category $\ModSfinf$
\[
\mu_\cG:\cM_\SG(\SGm^*(\cG))\to \bD^*(\cG)(S\to \okey).
\]
\end{enumerate}
\end{thm}

On the other hand, for any object $\SGm$ of the category $\ModSGffr$ or $\ModSGfinf$, we can define a dual object $\SGm^\vee$ which is compatible with Cartier duality of Barsotti-Tate groups or finite flat group schemes. In particular, for any object $\SGm$ of the category $\ModSGfinf$ killed by $p^n$, we have a commutative diagram of $G_{K_\infty}$-modules
\[
\xymatrix{
\cG(\okbar)\times \cG^\vee(\okbar)\ar[r]\ar@<-6ex>[d]^{\wr}_{\varepsilon_\cG}\ar@<+5ex>[d]^{\wr}_{\delta_\cG} & \bZ/p^n\bZ(1) \ar[d]\\
\TSG(\SGm^*(\cG))\times\TSG(\SGm^*(\cG)^\vee)\ar[r]& W_n(R),
}
\]
where the upper horizontal arrow is the pairing of Cartier duality, the lower horizontal arrow is a natural perfect pairing, $\delta_\cG$ is the composite
\[
\cG^\vee(\okbar)\overset{\varepsilon_{\cG^\vee}}{\simeq} \TSG(\SGm^*(\cG^\vee))\simeq \TSG(\SGm^*(\cG)^\vee)
\]
and the right vertical arrow is an injection (see \cite[Subsection 5.1]{Kim_2}, and also \cite[Proposition 4.4]{Ha_ramcorr}).

Let $\Gamma$ be a Barsotti-Tate group over $\okey$. We consider any element $g$ of $T_p(\Gamma)$ as a homomorphism $g:\bQ_p/\bZ_p\to \Gamma\times\Spec(\hat{\cO}_{\Kbar})$. By evaluating the map $\bD^*(g):\bD^*(\Gamma\times\Spec(\hat{\cO}_{\Kbar}))\to \bD^*(\bQ_p/\bZ_p)$ on the natural divided power thickening $\Acrys\to \hat{\cO}_{\Kbar}$, we obtain a homomorphism of $G_{K_\infty}$-modules
\begin{align*}
T_p(\Gamma)&\to \Hom_{S,\varphi, \Fil}(\bD^*(\Gamma)(\Acrys\to\hat{\cO}_{\Kbar}), \bD^*(\bQ_p/\bZ_p)(\Acrys\to\hat{\cO}_{\Kbar}))\\
&=\Tcrys(\bD^*(\Gamma)(S\to \okey)).
\end{align*}
This map is an injection, and an isomorphism after inverting $p$ (\cite[Theorem 7]{Fal}). Then we have the following compatibility of this map with the Breuil-Kisin classification.

\begin{lem}\label{lattice}
Let $\Gamma$ be a Barsotti-Tate group over $\okey$. Then the following diagram is commutative:
\[
\xymatrix{
T_p(\Gamma) \ar[r]^{\sim}_{\varepsilon_\Gamma}\ar@{^{(}->}[d] & T_\SG(\SGm^*(\Gamma))\ar@{^{(}->}[d]\\
\Tcrys(\bD^{*}(\Gamma)(S\to\okey))\ar[r]^-{\sim}_-{\Tcrys(\mu_\Gamma)} & \Tcrys(\cM_\SG(\SGm^*(\Gamma))).
}
\]
\end{lem}
\begin{proof}
Put $D=D^*(\Gamma)$ and $\SGm=\SGm^*(\Gamma)$. Consider the diagram
\[
\xymatrix{
T_p(\Gamma)\ar[r]\ar[rd] & \Tcrys(\bD^*(\Gamma)(S\to \okey)) \ar[r]^-{\sim}\ar[d] & \Tcrys(\cM_\SG(\SGm))\ar[ld] & \TSG(\SGm) \ar[l]\ar[lld]\\
& \Vcrys(D), & &
}
\]
where the left and the middle triangles are commutative by \cite[Theorem 5.6.2]{Kim_2} and Theorem \ref{KimBK} (\ref{KimBK1}), respectively. The commutativity of the right one is remarked in \cite[footnote 11]{Kim_2}. We briefly reproduce a proof of this remark for the convenience of the reader. We follow the notation of \cite{Ki_Fcrys}. In particular, let $\cO=\cO_{[0,1)}$ be the ring of rigid-analytic functions on the open unit disc over $K_0$ and $M=\cO\otimes_{\SG}\SGm$ be the associated $\varphi$-module over the ring $\cO$. We also put $\cD_0=(\cO[l_u]\otimes_{K_0}D)^{N=0}=\cO\otimes_{K_0}D$. Then the map $\TSG(\SGm)\to \Vcrys(D)$ is defined as the composite
\begin{align*}
\Hom_{\SG,\varphi}(\SGm,W(R))&\to \Hom_{\cO,\varphi}(M,\Bcrys^{+})\overset{(1\otimes\varphi)^*}\to \Hom_{\cO,\varphi}(\varphi^{*}M,\Bcrys^+)\\
&\overset{(1\otimes \xi)^*}{\to} \Hom_{\cO,\varphi, \Fil}(\cD_0, \Bcrys^+)\to \Hom_{K_0,\varphi,\Fil}(D,\Bcrys^+).
\end{align*}
Here the map $\xi: D\to M$ is the unique $\varphi$-compatible section and the map $1\otimes \xi: \cD_0=\cO\otimes_{K_0} D\to M$ factors through the injection 
\[
1\otimes \varphi: \varphi^*M=\cO\otimes_{\varphi,\cO}M\to M
\]
(\cite[Lemma 1.2.6]{Ki_Fcrys}). Put $\cD_\SG(\SGm)=\cM_\SG(\SGm)[1/p]=S_{K_0}\otimes_{\cO}\varphi^*M$. Then we have $K_0\otimes_{S_{K_0}}\cD_\SG(\SGm)=K_0\otimes_{\varphi,K_0}D$ and the composite
\[
s_0: K_0\otimes_{\varphi,K_0}D\overset{1\otimes\varphi}{\to}D\overset{\xi}{\to}\varphi^*M\to \cD_\SG(\SGm)
\]
is the unique $\varphi$-compatible section. Using this, we can check that the map $K_0\otimes_{\varphi,K_0}D\overset{1\otimes\varphi}{\to}D$ is an isomorphism of filtered $\varphi$-modules, where we consider on the left-hand side the induced filtration by the isomorphism 
\[
\cD_\SG(\SGm)/(\Fil^1 S)\cD_\SG(\SGm)\to K\otimes_{\varphi,K_0}D,
\]
and hence we can also check the above remark easily. Since the map $\varepsilon_\Gamma$ is defined by identifying the images of $T_p(\Gamma)$ and $\TSG(\SGm)$ in $\Vcrys(D)$, the lemma follows.
\end{proof}



\section{lower ramification subgroups}

In this section, we give a description of lower ramification subgroups of finite flat group schemes over $\okey$ in terms of the Breuil-Kisin classification. As an application, we derive Theorem \ref{ramcorr2} from this description. We begin with the following lemma, which gives upper bounds of the lower ramification of finite flat group schemes. For any valuation ring $V$ of height one with valuation $v$ and any $N$-tuple $\underline{x}=(x_1,\ldots,x_N)$ in $V$, we put $v(\underline{x})=\min_{l=1,\ldots,N} v(x_l)$.

\begin{lem}\label{lowramdeg}
\begin{enumerate}
\item\label{lowramdeg-triv} Let $\cK/\bQ_p$ be an extension of complete discrete valuation fields and $\cG$ be a finite flat group scheme over $\cO_\cK$ killed by some $p$-power. Then we have $\cG_i=0$ for any $i> 1/(p-1)$.
\item\label{lowramdeg-deg} Let $\cK$ be an extension of complete discrete valuation fields over $\bQ_p$ or $k((u))$ with valuation $v$ and $\cG$ be a finite flat generically etale group scheme over $\cO_\cK$ killed by some $p$-power. Then we have the following.
\begin{enumerate}
\item\label{lrd-a} $\cG_i=(\cG^0)_i$ for any $i>0$.
\item\label{lrd-b} $\cG_i=0$ for any $i> \deg(\cG)/(p-1)$. 
\end{enumerate}
Here $\cG_i$ and $\deg(\cG)$ are defined using $v$. Namely, we extend $v$ to a separable closure $\cK^{\sep}$ of $\cK$, write as $\omega_{\cG}\simeq \oplus_l \cO_{\cK}/(a_l)$ and put
\[
\cG_i(\cO_{\cK^{\sep}})=\Ker(\cG(\cO_{\cK^{\sep}})\to \cG(\cO_{\cK^{\sep},i})),\quad \deg(\cG)=\sum_l v(a_l).
\]
\end{enumerate}
\end{lem}
\begin{proof}
For the assertion (\ref{lowramdeg-triv}), we may replace $\cK$ by its finite extension and assume $\cG^\vee(\cO_{\bar{\cK}})=\cG^\vee(\cO_{\cK})$ for an algebraic closure $\bar{\cK}$ of $\cK$. By Cartier duality, there exists a generic isomorphism $\cG\to \cG'=\oplus_l \mu_{p^{n_l}}$ for some $n_l$. Then $\cG'_i=0$ for any $i>1/(p-1)$ and the assertion follows from the commutative diagram
\[
\xymatrix{
\cG(\cO_{\bar{\cK}}) \ar[r]_{\sim}\ar[d] & \cG'(\cO_{\bar{\cK}})\ar[d] \\
\cG(\cO_{\bar{\cK},i}) \ar[r] & \cG'(\cO_{\bar{\cK},i}).
}
\]

Let us consider the assertion (\ref{lowramdeg-deg}). For any $i>0$, we have a commutative diagram
\[
\xymatrix{
0 \ar[r] & \cG^0(\cO_{\cK^{\sep}}) \ar[r] & \cG(\cO_{\cK^{\sep}}) \ar[r]\ar[d] & \cG^{\et}(\cO_{\cK^{\sep}}) \ar[r]\ar[d]  & 0\\
 & & \cG(\cO_{\cK^{\sep},i}) \ar[r] & \cG^{\et}(\cO_{\cK^{\sep},i}), & &
}
\]
where the upper row is the connected-etale sequence. Then the right vertical arrow is an isomorphism and the part (\ref{lrd-a}) follows. 

For the part (\ref{lrd-b}), suppose $i>\deg(\cG)/(p-1)$. By the part (\ref{lrd-a}), we may assume that $\cG$ is connected. By \cite[Proposition 1.5]{Ti_HN}, we have a presentation of the affine algebra $\cO_\cG$ of $\cG$
\begin{align*}
\cO_{\cG}&\simeq \cO_\cK[[X_1,\ldots,X_d]]/(f_1,\ldots,f_d),\\
(f_1,\ldots,f_d)&\equiv (X_1,\ldots,X_d)U \mod\deg p
\end{align*}
with some $U\in M_d(\cO_\cK)$ satisfying the equality $v(\det(U))=\deg(\cG)$, where $X_1=\cdots=X_d=0$ gives the zero section. Let $\hat{U}$ be the matrix satisfying $U\hat{U}=\det(U)I_d$, where $I_d$ is the identity matrix. For any element $\underline{x}=(x_1,\ldots,x_d)$ of $\cG(\cO_{\cK^{\sep}})$, multiplying by $\hat{U}$ implies the inequality
\[
v(\underline{x})+v(\det(U))\geq p v(\underline{x}).
\]
Thus we obtain the inequality $v(\underline{x})\leq \deg(\cG)/(p-1)$ unless $\underline{x}=0$ and the assertion follows.
\end{proof}

For any positive rational number $i\leq 1$, we let $W_n^\PD(R)_i$ denote the divided power envelope of the composite 
\[
\theta_{n,i}:W_n(R)\overset{\theta_n}{\to} \cO_{\Kbar,n}\to \cO_{\Kbar,i}
\]
compatible with the canonical divided power structure on the ideal $pW_n(R)$. Note that, by fixing a generator $\up^i$ of the principal ideal $m_R^{\geq i}$, we have an isomorphism of $R$-algebras
\begin{equation}\label{Ypol}
W_n(R)[Y_1,Y_2,\ldots]/([\up^{i}]^p-p Y_1,Y_1^p-p Y_2, Y_2^p-p Y_3,\ldots) \to W_n^\PD(R)_i
\end{equation}
sending $Y_l$ to $\delta^l([\up^{i}])$, where we put $\delta(x)=(p-1)!\gamma_p(x)$ with the $p$-th divided power $\gamma_p$. The surjection $\theta_{n,i}$ defines a divided power thickening $W_n^\PD(R)_i\to \cO_{\Kbar,i}$ over the thickening $S\to \okey$, which is denoted by $A_{n,i}$. Put
\[
I_{n,i}=\Ker(W_n(R)\overset{\varphi}{\to} W_n^\PD(R)_i). 
\]
From the definition, we see the inclusion $I_{n,i}\subseteq I_{n,i'}$ for any $i>i'$. Then the main theorem of this section is the following.
\begin{thm}\label{lowramcorrn}
Let $\cG$ be a finite flat group scheme over $\okey$ killed by $p^n$ and $\SGm=\SGm^*(\cG)$ be the corresponding object of the category $\ModSGfinf$. Then the natural isomorphism 
\[
\varepsilon_\cG: \cG(\okbar)\to \TSG(\SGm)=\Hom_{\SG,\varphi}(\SGm,W_n(R))
\]
induces an isomorphism
\[
\cG_i(\okbar) \simeq \Hom_{\SG,\varphi}(\SGm,I_{n,i})
\]
for any positive rational number $i\leq 1$.
\end{thm}

Note that Theorem \ref{ramcorr2} follows from Theorem \ref{lowramcorrn}. Indeed, by Cartier duality, a theorem of Tian and Fargues (\cite[Theorem 1.6]{Ti} or \cite[Proposition 6]{Fa}) and \cite[Theorem 3.3]{Ha_ramcorr}, it is enough to show the assertion of Theorem \ref{ramcorr2} on lower ramification subgroups. Moreover, since the $i$-th lower ramification subgroups of $\cG$ and $\cH(\SGm^*(\cG))$ vanish for any $i>1/(p-1)$ (\cite[Corollary 3.5 and Remark 3.6]{Ha_ramcorr}), we may assume $i\leq 1$. Then the equality $I_{1,i}=m_R^{\geq i}$ and Theorem \ref{lowramcorrn} imply Theorem \ref{ramcorr2}.

We show Theorem \ref{lowramcorrn} by relating both sides of the isomorphism via Breuil modules using the lemma below.

\begin{lem}\label{modinj}
Let $i\leq 1$ be a positive rational number and $\cG$ be a finite flat group scheme over $\cO_{K,i}$ killed by $p^n$. Then the map
\begin{align*}
\cG(\cO_{\Kbar,i})=\Hom_{\cO_{\Kbar,i}}(\bZ/p^n\bZ,\cG\times \bar{\sS}_i)&\to \Hom(\bD^*(\cG)(A_{n,i}), \bD^*(\bZ/p^n\bZ)(A_{n,i}))\\
&=\Hom(\bD^*(\cG)(A_{n,i}), W_n^\PD(R)_i)
\end{align*}
defined by $g\mapsto \bD^*(g)(A_{n,i})$ is an injection.
\end{lem}
\begin{proof}
Suppose that a homomorphism $g:\bZ/p^n \bZ\to \cG\times\bar{\sS}_i$ satisfies $\bD^*(g)(A_{n,i})=0$. We can take a finite extension $L/K$ such that the map $g$ is defined over $\Spec(\cO_{L,i})$. Then we have the commutative diagram
\[
\xymatrix{
\Hom_{\cO_{L,i}}(\bZ/p^n\bZ,\cG\times\sS_{L,i})\ar[r]\ar[d] & \Hom(\bD^*(\cG\times \sS_{L,i})(A_{n,i}), \bD^*(\bZ/p^n\bZ)(A_{n,i}))\ar[d]_{\wr} \\
\Hom_{\cO_{\Kbar,i}}(\bZ/p^n\bZ,\cG\times\bar{\sS}_i)\ar[r] & \Hom(\bD^*(\cG\times \bar{\sS}_i)(A_{n,i}), \bD^*(\bZ/p^n\bZ)(A_{n,i}))
}
\]
and thus we may assume $L=K$. 

Put $\Sigma=\Spec(\bZ_p)$ and $\Sigma_n=\Spec(\bZ/p^n\bZ)$. Consider the big fppf crystalline site $\CRYS(\sS_i/\Sigma)$ and its topos $(\sS_i/\Sigma)_\CRYS$ (\cite{BBM}). Note that the local ring $\cO_{K,i}$ is a Noetherian complete intersection ring and, for any finite extension $L/K$, the ring $\cO_{L,i}$ is faithfully flat and of relative complete intersection over $\cO_{K,i}$. Thus, by \cite[Proposition 1.2 and Lemma 4.1]{deJM}, we see that the composite
\begin{align*}
\Hom_{\cO_{K,i}}(\bZ/p^n\bZ,\cG) &\to \Hom_{(\sS_i/\Sigma)_\CRYS}(\bD^*(\cG),\bD^*(\bZ/p^n\bZ))\\
&\to \Hom_{(\bar{\sS}_i/\Sigma)_\CRYS}(\bD^*(\cG),\bD^*(\bZ/p^n\bZ))
\end{align*}
is an injection. 

Consider the natural morphism of topoi
\[
i_{n\CRYS}: (\bar{\sS}_i/\Sigma_n)_\CRYS \to (\bar{\sS}_i/\Sigma)_\CRYS.
\]
Since the crystal $\bD^*(\bZ/p^n\bZ)$ is isomorphic to the quotient $\cO_{\bar{\sS}_i/\Sigma}/p^n \cO_{\bar{\sS}_i/\Sigma}$ of the structure sheaf $\cO_{\bar{\sS}_i/\Sigma}$ (\cite[Exemples 4.2.16]{BBM}) and this is equal to $i_{n\CRYS*}(\cO_{\bar{\sS}_i/\Sigma_n})$ (\cite[(4.2.17.4)]{BBM}), the natural map
\begin{align*}
i_{n\CRYS}^*: &\Hom_{(\bar{\sS}_i/\Sigma)_\CRYS}(\bD^*(\cG),\bD^*(\bZ/p^n\bZ))\\
&\to \Hom_{(\bar{\sS}_i/\Sigma_n)_\CRYS}(i_{n\CRYS}^*(\bD^*(\cG)),i_{n\CRYS}^*(\bD^*(\bZ/p^n\bZ)))
\end{align*}
is an isomorphism. 

Finally, we claim that the thickening $A_{n,i}$ defines the final object of the big crystalline site $\CRYS(\bar{\sS}_i/\Sigma_n)$. Indeed, it suffices to show that for any $\cO_{\Kbar,i}$-algebra $\cO_U$, any $\bZ/p^n\bZ$-algebra $\cO_T$ and any surjection $\cO_T\to \cO_U$ defined by a divided power ideal $J_T$, the composite 
\[
W_n(R) \overset{\theta_{n,i}}{\to} \cO_{\Kbar,i}\to \cO_U
\]
uniquely factors through $\cO_T$. For this, we define the map $f:W_n(R)\to \cO_T$ as follows: For any element $r=(r_0,\ldots,r_{n-1})$ of the ring $W_n(R)$, choose a lift $\widehat{\prjt_n(r_l)}$ in $\cO_T$ of the element $\prjt_n(r_l)$ for any $l=0,\ldots,n-1$ and put
\[
f(r)=\sum_{l=0}^{n-1} p^l(\widehat{\prjt_n(r_l)})^{p^{n-l}}.
\]
This is independent of the choice of lifts and gives a ring homomorphism satisfying the condition. Conversely, suppose that a homomorphism $f': W_n(R)\to \cO_T$ satisfies the condition. Then, for any element $r=(r_0,\ldots,r_{n-1})$ of the ring $W_n(R)$, we have $f'(r)=\sum_{l=0}^{n-1}p^l f'([r_l]^{1/p^n})^{p^{n-l}}$ and $f'([r_l]^{1/p^n})\mod J_T= \prjt_n(r_l)$. Thus the uniqueness follows. Hence the evaluation map on the thickening $A_{n,i}$
\begin{align*}
\Hom_{(\bar{\sS}_i/\Sigma_n)_\CRYS}&(i_{n\CRYS}^*(\bD^*(\cG)),i_{n\CRYS}^*(\bD^*(\bZ/p^n\bZ)))\\
&\to \Hom(\bD^*(\cG)(A_{n,i}), W_n^\PD(R)_i)
\end{align*}
is an injection. This concludes the proof of the lemma.
\end{proof}

\begin{proof}[Proof of Theorem \ref{lowramcorrn}]
Take a resolution of $\cG$ by Barsotti-Tate groups over $\okey$
\[
0\to \cG \to \Gamma_1 \to \Gamma_2 \to 0
\]
and consider the associated exact sequence of Kisin modules
\[
0 \to \SGn_2 \to \SGn_1 \to \SGm \to 0.
\]
Put $\cM=\cM_\SG(\SGm)$ and $\cN_l=\cM_\SG(\SGn_l)$ for $l=1,2$. By Lemma \ref{lattice} and the definition of the anti-equivalence $\SGm^*(-)$, we have a diagram 
\[
\xymatrix{
T_p(\Gamma_1) \ar@/^1pc/@{.>}[rr]^>>>>>>>>>>>{\varepsilon_{\Gamma_1}}\ar@{^{(}->}[r] \ar[d] & \Tcrys(\cN_1) \ar[d] & \TSG(\SGn_1) \ar@{_{(}->}[l]\ar[d]\\
T_p(\Gamma_2) \ar@/^1pc/@{.>}[rr]^>>>>>>>>>>>{\varepsilon_{\Gamma_2}}\ar@{^{(}->}[r] \ar[d]_{\pi_\cG} & \Tcrys(\cN_2) \ar[d]_{\pi_\cM} & \TSG(\SGn_2) \ar@{_{(}->}[l]\ar[d]_{\pi_\SGm}\\
\cG(\okbar) \ar@/^1pc/@{.>}[rr]^>>>>>>>>>>>{\varepsilon_\cG}\ar[r]\ar[d] & \Hom_{S,\varphi}(\cM,W_n^\PD(R)) \ar[d]         & \TSG(\SGm) \ar[l]\ar[d]\\
\cG(\okbari) \ar@{^{(}->}[r] & \Hom_{S,\varphi}(\cM,W_n^\PD(R)_i) & \Hom_{\SG,\varphi}(\SGm,W_n(R)/I_{n,i}) \ar@{_{(}->}[l],
}
\]
where the left horizontal arrows are induced by $g\mapsto \bD^{*}(g)$ and the right horizontal arrows are the maps sending $f$ to $1\otimes (\varphi\circ f)$. The middle left vertical arrow $\pi_\cG: T_p(\Gamma_2) \to \cG(\okbar)$ is defined as follows: For $g\in T_p(\Gamma_2)$, the element $p^ng$ is contained in the image of $T_p(\Gamma_1)=\varprojlim_{l} \Gamma_1[p^l](\okbar)$ and put $p^ng=h=(h_n)_{n>0}$. Then the element $h_n\in\Gamma_1[p^n](\okbar)$ is contained in the subgroup $\cG(\okbar)$ and the map $\pi_\cG$ is defined by $g\mapsto h_n$. We define the map $\pi_\cM: \Tcrys(\cN_2) \to \Hom_{S,\varphi}(\cM,W_n^\PD(R))$ similarly: For any map $f:\cN_2 \to\Acrys$, the map $p^nf$ induces a map $\cN_1\to \Acrys$. Its composite with the natural map $\Acrys \to W_n^\PD(R)$ factors through $\cM$ and defines the map $\pi_\cM(f):\cM\to W_n^\PD(R)$. The map $\pi_\SGm$ is defined in the same way. From these definitions, we see that the diagram is commutative. Note that the bottom left horizontal arrow is an injection by Lemma \ref{modinj}, and that the bottom right horizontal arrow is also an injection by the definition of the ideal $I_{n,i}$. 

Thus, for any element $g\in \cG(\okbar)$, its image in $\cG(\cO_{\Kbar,i})$ is zero if and only if the image of $\varepsilon_\cG(g)\in \TSG(\SGm)$ in $\Hom_{\SG,\varphi}(\SGm,W_n(R)/I_{n,i})$ is zero. Hence the theorem follows. 
\end{proof}

Next we give an explicit description of the ideal $I_{n,i}$. We identify the rings of both sides of the isomorphism (\ref{Ypol}).

\begin{prop}\label{Y}
Let $n_1,\ldots,n_l$ be integers satisfying $0\leq n_j \leq p-1$ for any $j$ and $r$ be an element of $W_n(R)$. If the element $rY_1^{n_1}\cdots Y_l^{n_l}$ is zero in the ring $W_n^\PD(R)_i$, then $[\up^{i}]^p| r$ in the ring $W_n(R)$. In particular, we have the inclusion $I_{n,i}\subseteq ([\up^{i}])$.
\end{prop}
\begin{proof}
It suffices to show that the equality in the ring $W_n(R)[Y_1,\ldots,Y_l]$
\begin{equation}\label{rY}
r Y_1^{n_1}\cdots Y_l^{n_l}=([\up^{i}]^p-p Y_1)f_0+(Y_1^p-p Y_2)f_1+\cdots+(Y_{l-1}^p-p Y_l)f_{l-1}+Y_l^pf_l
\end{equation}
with $f_0,\ldots, f_l$ in this ring implies $[\up^{i}]^p|r$. By replacing $f_j$'s, we may assume the inequality
\begin{equation}\label{assump}
\deg_{j'}(f_j)<p\quad (j'=j+1,\ldots,l),
\end{equation}
where $\deg_{j'}$ means the degree with respect to $Y_{j'}$.

For any $l$-tuple $\um=(m_1,\ldots,m_l)$, write $\uY^{\um}=Y_1^{m_1}\cdots Y_l^{m_l}$ and let $c_{j,\um}$ be the coefficient of $\uY^{\um}$ in $f_j$. Put $\un=(n_1,\ldots,n_l)$ and $e_j=(0,\ldots,0,1,0,\ldots,0)$ with $1$ on the $j$-th entry. We consider the lexicographic order on the module $\bZ^l$: we say $\um < \um'$ if there exists $j$ with $1\leq j \leq l$ such that $m_j<m'_j$ and $m_{j'}=m'_{j'}$ for any $j<j'\leq l$. Then we have the equality
\[
r\uY^{\un}=[\up^{i}]^p c_{0,\un}\uY^{\un}+\sum_{j=0}^{l-1}(-p Y_{j+1}) c_{j,\un-e_{j+1}}\uY^{\un-e_{j+1}}.
\]

Now we claim that
\begin{equation}\label{czero}
c_{j,\un-e_{j+1}}= 0 \quad (j=0,\ldots, l-1).
\end{equation}
Suppose the contrary. Choose $j$ such that $0\leq j\leq l-1$ and $c_{j,\un-e_{j+1}}\neq 0$. Consider the term $c_{j,\un-e_{j+1}}\uY^{\un-e_{j+1}}$ in $f_j$. The right-hand side of the equality (\ref{rY}) contains the term $c_{j,\un-e_{j+1}}\uY^{\un+p e_j-e_{j+1}}$ for $j\geq 1$ and $[\up^i]^p c_{0,\un-e_1}\uY^{\un-e_1}$ for $j=0$. Since $\un+pe_j-e_{j+1}<\un$ and $\un-e_1<\un$, the equation (\ref{rY}) and the assumption (\ref{assump}) imply the equation
\begin{align*}
c_{j,\un-e_{j+1}}&\uY^{\un+p e_j-e_{j+1}}\\
&=-\sum_{j'=j-1}^{l-1}(-p Y_{j'+1}) c_{j',\un+p e_j - e_{j+1} -e_{j'+1}}\uY^{\un+p e_j - e_{j+1} -e_{j'+1}}
\end{align*}
for $j\geq 1$ and
\[
[\up^i]^p c_{0,\un-e_1}\uY^{\un-e_1}=-\sum_{j'=0}^{l-1}(-p Y_{j'+1}) c_{j',\un-e_1 -e_{j'+1}}\uY^{\un- e_1 -e_{j'+1}}
\]
for $j=0$. We let $\mathrm{Eq}(1)$ denote this equation. Put $\um(1)=\un+p e_j-e_{j+1}$ for $j\geq 1$ and $\um(1)=\un-e_1$ for $j=0$. Repeating this by arbitrarily choosing a term with nonzero coefficient $c_{j',\um'}$ on the right-hand side of the equation $\mathrm{Eq}(s)$, we obtain a series of equations $\mathrm{Eq}(1), \mathrm{Eq}(2),\ldots$ and a sequence of $l$-tuples of non-negative integers $\um(1), \um(2),\ldots$ such that $\mathrm{Eq}(s)$ is an equation of monomials of degree $\um(s)$ for any $s\geq 1$. Note that if there is no such term on the right-hand side of the equation $\mathrm{Eq}(s)$, the procedure stops. On the other hand, if the equation $\mathrm{Eq}(s)$ is either of the types
\[
c \uY^{\um(s)}=
\left\{
\begin{array}{ll}
-\cdots-(Y_j^p)c_{j,\um(s)-p e_j} \uY^{\um(s)-p e_j}-\cdots & (1\leq j\leq l-1)\\
-[\up^i]^p c_{0,\um(s)}\uY^{\um(s)}-\cdots & (j=0)
\end{array}
\right.
\]
with some $c\in W_n(R)$ such that the indicated term is chosen and that $c_{j,\um(s)-p e_j}$ ({\it resp.} $c_{0,\um(s)}$) is contained in the ideal $p^{n-1}W_n(R)$, then the equation $\mathrm{Eq}(s+1)$ is empty and the procedure also stops. In the latter case, we put $\um(s+1)=\um(s)-p e_j +e_{j+1}$ for $1\leq j\leq l-1$ and $\um(s+1)=\um(s)+e_1$ for $j=0$.

\begin{lem}\label{lex}
The sequence $\um(s)$ is strictly decreasing with respect to the lexicographic order on $\bZ^l$. 
\end{lem}
\begin{proof}
Note the inequalities $\un>\um(1)>\um(2)$. Suppose that we have $\um(1)>\um(2)>\cdots>\um(t)\leq \um(t+1)$ for some $t\geq 2$. Then the term $Y^p_l f_l$ in the equality (\ref{rY}) does not affect the equation $\mathrm{Eq}(s)$ for $1\leq s\leq t$. Thus, by the construction, one of the following four cases holds for each $1\leq s\leq t$:
\[
\left.
\begin{array}{ll}
(C_j) &\um(s+1)=\um(s)+p e_j-e_{j+1}\text{ for some }1\leq j\leq l-1,\\
(C'_j)&\um(s+1)=\um(s)-p e_j+e_{j+1}\text{ for some }1\leq j\leq l-1,\\
(C_0) & \um(s+1)=\um(s)-e_1,\\
(C'_0)& \um(s+1)=\um(s)+e_1.
\end{array}
\right.
\]
Moreover, $(C_j)$ and $(C'_j)$ do not occur consecutively for any $j$ satisfying $0\leq j \leq l-1$. Note the inequality $\um(s)>\um(s+1)$ for $(C_j)$ and $\um(s)<\um(s+1)$ for $(C'_j)$.

First we claim that $(C'_0)$ does not hold for $s=t$. Suppose the contrary. Then $(C_j)$ holds for $s=t-1$ with some $j$ satisfying $1\leq j \leq l-1$. Hence the $j$-th entry $m(t)_j$ of the $l$-tuple $\um(t)$ is no less than $p$. The equation $\mathrm{Eq}(t)$
\[
c_{j,\um(t-1)-e_{j+1}}\uY^{\um(t)}=-[\up^i]^p c_{0,\um(t)}\uY^{\um(t)}-\cdots
\]
implies $\deg_j(f_0)\geq p$. This contradicts the assumption (\ref{assump}).

Hence $(C'_j)$ holds for $s=t$ with some $1\leq j \leq l-1$. From this we see the inequality $m(t)_j\geq p$. Since $n_j<p$, there exists an integer $t'$ with $1\leq t' \leq t-2$ such that $(C_j)$ holds for $s=t'$ and that it does not hold for any $s$ satisfying $t'<s\leq t$.

Next we claim the equality $m(s)_j=m(t')_j+p$ for any $s$ satisfying $t'< s \leq t$. Suppose the contrary and take the smallest integer $t''$ with $t'< t''<t$ such that $(C_{j-1})$ holds  for $s=t''$. Then $m(s)_j=m(t')_j+p$ for $t'< s\leq t''$ and $m(t''+1)_j=m(t')_j+p-1$. By assumption, we also have the inequality $m(t''+1)_j\geq m(t)_j\geq p$. On the other hand, the equation $\mathrm{Eq}(t'')$ is
\[
c\uY^{\um(t'')}=-\cdots-(-p Y_j)c_{j-1, \um(t'')-e_j} \uY^{\um(t'')-e_j}-\cdots
\]
with some $c\in W_n(R)$. Hence we obtain 
\[
\deg_j(f_{j-1})\geq m(t'')_j-1=m(t')_j+p-1\geq p,
\]
which contradicts the assumption (\ref{assump}).

Now let $j_0$ be the non-negative integer such that $(C_{j_0})$ holds for $s=t-1$. Then $j_0\neq j,j-1$ by the constancy of $m(s)_j$ which we have just proved. The equation $\mathrm{Eq}(t-1)$ is
\[
c\uY^{\um(t-1)}=-\cdots- (-p Y_{j_0+1}) c_{j_0, \um(t-1)-e_{j_0+1}} \uY^{\um(t-1)-e_{j_0+1}} -\cdots
\]
with some $c\in W_n(R)$ and thus $\deg_j(f_{j_0})\geq m(t-1)_j= m(t')_j+p \geq p$. By the assumption (\ref{assump}), we obtain the inequality $j_0>j$. In particular, we have $j_0\geq 1$ and $\um(t)=\um(t-1)+p e_{j_0}-e_{j_0+1}$. Therefore the equation $\mathrm{Eq}(t)$ is
\[
c'\uY^{\um(t)}=-\cdots - (Y_j^p) c_{j,\um(t)-p e_j} \uY^{\um(t)-p e_j}-\cdots
\]
with some $c'\in W_n(R)$ and $\deg_{j_0}(f_j)\geq m(t)_{j_0}\geq p$. This contradicts the assumption (\ref{assump}) and the lemma follows.
\end{proof}

By Lemma \ref{lex}, the case $(C'_j)$ does not occur in the procedure for any non-negative integer $j$. In particular, if there is no term with non-zero $c_{j',\um'}$ on the right-hand side of the equation $\mathrm{Eq}(s)$ for some $s$, then the equation is
\[
[\up^i]^{p \epsilon} c_{j'',\um''} \uY^{\um(s)}=0,
\]
where $c_{j'',\um''}\uY^{\um''}$ is the chosen term on the right-hand side of the equation $\mathrm{Eq}(s-1)$ and $\epsilon\in\{0,1\}$. Note that this occurs for $s$ satisfying $\um(s)=(0,\ldots,0)$, since in this case $(C_0)$ holds for $s-1$. Therefore, Lemma \ref{lex} implies that, for any choice of terms as above, we end up with an equation of this type for a sufficiently large $s$. Since the element $[\up^i]^p$ is a non-zero divisor in the ring $W_n(R)$, we see the equality $c_{j'',\um''}=0$. This contradicts the choice of terms and the equality (\ref{czero}) follows.

Hence we obtain the equality
\[
r\uY^{\un}=[\up^{i}]^p c_{0,\un}  \uY^{\un}
\]
and thus $[\up^i]^p|r$. This concludes the proof of Proposition \ref{Y}.
\end{proof}

\begin{lem}\label{I}
Put $\mathbf{n}(s)=v_p((ps)!)$ for any non-negative integer $s$. Then an element $r=(r_0,\ldots,r_{n-1})$ of the ring $W_n(R)$ is contained in the ideal $I_{n,i}$ if and only if the condition
\begin{equation}\label{condI}
[\up^{i}]^s| (r_0,\ldots, r_{n-1-\mathbf{n}(s-1)},0,\ldots,0)
\end{equation}
holds for any $s\geq 1$.
\end{lem}
\begin{proof}
Let $r$ be an element of the ideal $I_{n,i}$ and show the condition (\ref{condI}) for $r$ by induction on $s$. The case of $s=1$ follows from Proposition \ref{Y}. Suppose that the condition (\ref{condI}) holds for some $s\geq 1$. Let $r'=(r'_0,\ldots,r'_{n-1-\mathbf{n}(s-1)},0,\ldots,0)$ be the element of $W_n(R)$ such that
\[
(r_0,\ldots,r_{n-1-\mathbf{n}(s-1)},0,\ldots,0)=[\up^{i}]^s r'.
\]
We write the $p$-adic expansion of the integer $s$ as 
\[
s=n_1+p n_2+\cdots+p^{l}n_{l}
\]
with $0\leq n_{j} \leq p-1$. 
Then we have the equality in the ring $W_n^\PD(R)_i$
\[
\varphi(r)=p^{\mathbf{n}(s)}\varphi(r')Y_1^{n_1}\cdots Y_{l}^{n_{l}}
\]
and Proposition \ref{Y} implies that $[\up^i]$ divides $p^{\mathbf{n}(s)}r'$. Hence the element $[\up^i]$ divides $(r'_0,\ldots,r'_{n-1-\mathbf{n}(s)},0,\ldots,0)$ and thus
\[
[\up^{i}]^{s+1} | (r_0,\ldots,r_{n-1-\mathbf{n}(s)},0,\ldots,0).
\]
Conversely, suppose that an element $r$ of the ring $W_n(R)$ satisfies the condition (\ref{condI}) for any $s\geq 1$. Since we have the inequality $\mathbf{n}(s)\geq n$ for some $s$, a similar argument as above shows the equality $\varphi(r)=0$ in the ring $W_n^\PD(R)_i$.  This concludes the proof of the lemma.
\end{proof}

\begin{rmk}
Lemma \ref{I} enables us to compute the ideal $I_{n,i}$. For example, $I_{2,i}=(m_R^{\geq 2i},m_R^{\geq pi})\subseteq W_2(R)$ and 
\[
I_{3,i}=\left\{
\begin{array}{l}
(m_R^{\geq 2 i}, m_R^{\geq 4 i}, m_R^{\geq 4 i})\quad (p=2),\\
(m_R^{\geq 3 i}, m_R^{\geq  2p i}, m_R^{\geq p^2 i})\quad (p\geq 3).
\end{array}
\right.
\] 
\end{rmk}

Finally we prove a relationship between the ideals $I_{n-1,pi}$ and $I_{n,i}$, which will be used in Section \ref{Rab}.

\begin{lem}\label{ipi}
For any $r=(r_0,\ldots,r_{n-2})\in I_{n-1,pi}$ and $r_{n-1}\in R$, we have
\[
\hat{r}=(r_0,\ldots,r_{n-2},\up^{ip^{n-1}}r_{n-1})\in I_{n,i}.
\]
\end{lem}
\begin{proof}
By Lemma \ref{I}, we have 
\[
[\up^{p i}]^{s} | (r_0,\ldots, r_{n-2-\mathbf{n}(s-1)},0,\ldots,0)
\]
in the ring $W_{n-1}(R)$ for any $s\geq 1$ satisfying $\mathbf{n}(s-1)<n-1$. Let us show that the element $\hat{r}=(\hat{r}_0,\ldots,\hat{r}_{n-1})$ satisfies the condition
\[
[\up^i]^s | (\hat{r}_0,\ldots, \hat{r}_{n-1-\mathbf{n}(s-1)},0,\ldots,0)
\]
in the ring $W_n(R)$ for any $s\geq 1$ satisfying $\mathbf{n}(s-1)<n$. The case of $s=1$ follows from the definition of $\hat{r}$. Suppose $s\geq 2$. Since $\mathbf{n}(s-2)+1\leq \mathbf{n}(s-1)$, we have $n-1-\mathbf{n}(s-1)\leq n-2-\mathbf{n}(s-2)$ and $[\up^{p i}]^{s-1}$ divides $(\hat{r}_0,\ldots, \hat{r}_{n-1-\mathbf{n}(s-1)})$. Then the inequality $p(s-1)\geq s$ implies the condition. This concludes the proof of the lemma.
\end{proof}



\section{Application to canonical subgroups}\label{Rab}

In this section, we prove Theorem \ref{canlow} and Theorem \ref{familycan}. First we consider Theorem \ref{canlow}. Let $K/\bQ_p$ be an extension of complete discrete valuation fields. Let $\cG$ be a truncated Barsotti-Tate group of level $n$, height $h$ and dimension $d$ over $\okey$ with $0<d<h$ and Hodge height $w<(p-1)/p^n$. Let $\cC_n$ be the level $n$ canonical subgroup of $\cG$ as in \cite[Theorem 1.1]{Ha_cansubZ}. By a base change argument and the uniqueness of $\cC_n$ (\cite[Proposition 3.8]{Ha_cansubZ}), we may assume that the residue field $k$ is perfect.

Let $\SGm=\SGm^*(\cG)$ be the corresponding object of the category $\ModSGfinf$. Then, by \cite[Remark 3.4]{Ha_cansubZ}, the object $\SGm/p\SGm$ has a basis $\bar{e}_1,\ldots,\bar{e}_h$ such that
\[
\varphi(\bar{e}_1,\ldots,\bar{e}_h)=(\bar{e}_1,\ldots,\bar{e}_h)\begin{pmatrix} P_1 & P_2 \\ u^e P_3 & u^e P_4\end{pmatrix},
\]
where the matrices $P_i$ have entries in the ring $k[[u]]$ with $P_1\in M_{h-d}(k[[u]])$, $v_R(\det(P_1))=w$ and $\begin{pmatrix} P_1& P_2\\ P_3 &P_4\end{pmatrix}\in GL_h(k[[u]])$. Let $\hat{P}_1$ be the element of $M_{h-d}(k[[u]])$ such that $P_1\hat{P}_1=u^{e w} I_{h-d}$. Let $B$ be the unique solution in $M_{d,h-d}(k[[u]])$ of the equation
\[
B=P_3\hat{P}_1-u^{ep(1-w)-ew}BP_2\varphi(B)\hat{P}_1+u^{ep(1-w)}P_4\varphi(B)\hat{P}_1
\]
and put $D=P_1+u^{ep(1-w)}P_2\varphi(B)$, which also satisfies $v_R(\det(D))=w$ (see the proof of \cite[Lemma 3.3]{Ha_cansub}). Moreover, put
\[
(\bar{e}'_1,\ldots,\bar{e}'_{h-d})=(\bar{e}_1,\ldots,\bar{e}_h)\begin{pmatrix}I_{h-d} \\u^{e(1-w)}B\end{pmatrix}.
\]
The elements $\bar{e}'_1,\ldots,\bar{e}'_{h-d}, \bar{e}_{h-d+1},\ldots,\bar{e}_h$ form a basis of the $\SG_1$-module $\SGm/p\SGm$ satisfying
\[
\varphi(\bar{e}'_1,\ldots,\bar{e}'_{h-d},\bar{e}_{h-d+1},\ldots,\bar{e}_h)=(\bar{e}'_1,\ldots,\bar{e}'_{h-d},\bar{e}_{h-d+1},\ldots,\bar{e}_h)\begin{pmatrix}D &P_2 \\ 0 & u^{e(1-w)}P'_4\end{pmatrix}
\]
for some matrix $P'_4\in M_d(k[[u]])$. 
Then we have the following description of the level one canonical subgroup $\cC_1$ of $\cG[p]$.

\begin{lem}\label{canHdg1}
Let $f$ be an element of the module $\Hom_{\SG,\varphi}(\SGm/p\SGm,R)$ defined by
\[
(\bar{e}_1,\ldots,\bar{e}_h)\mapsto (\underline{x},\underline{y})
\]
with an $(h-d)$-tuple $\underline{x}$ and a $d$-tuple $\underline{y}$ in $R$. Then $f$ corresponds to an element of $\cC_1(\okbar)$ by the isomorphism
\[
\varepsilon_{\cG[p]}: \cG[p](\okbar)\simeq \Hom_{\SG,\varphi}(\SGm/p\SGm,R)
\]
if and only if $v_R(\underline{x}+u^{e(1-w)}\underline{y}B)>w/(p-1)$. 
\end{lem}
\begin{proof}
Let $\SGl$ be the $\SG_1$-submodule of $\SGm/p\SGm$ generated by $\bar{e}'_1,\ldots,\bar{e}'_{h-d}$. Then $\SGl$ defines a subobject of $\SGm/p\SGm$ in the category $\ModSGf$. Put $\SGn=(\SGm/p\SGm)/\SGl$. By \cite[Lemma 3.2 and Theorem 3.5 (1)]{Ha_cansubZ}, the level one canonical subgroup $\cC_1$ is the closed subgroup scheme of $\cG[p]$ corresponding to the object $\SGn$. We have the commutative diagram
\[
\xymatrix{
0\ar[r] & \cC_1(\okbar) \ar[r]\ar[d]^{\varepsilon_{\cC_1}}_{\wr} & \cG[p](\okbar)\ar[r]\ar[d]^{\varepsilon_{\cG[p]}}_{\wr} & (\cG[p]/\cC_1)(\okbar) \ar[r]\ar[d]^{\varepsilon_{\cG[p]/\cC_1}}_{\wr} & 0 \\
0 \ar[r] & \Hom_{\SG,\varphi}(\SGn,R) \ar[r] & \Hom_{\SG,\varphi}(\SGm/p\SGm,R)\ar[r]^-{\iota^*}& \Hom_{\SG,\varphi}(\SGl,R)\ar[r] & 0,
}
\]
where the rows are exact and the vertical arrows are isomorphisms. The element $f$ corresponds to an element of $\cC_1(\okbar)$ if and only if $\iota^*(f)=0$. The map $\iota^*(f):\SGl\to R$ is defined by
\[
(\bar{e}'_1,\ldots,\bar{e}'_{h-d})\mapsto \underline{x}+u^{e(1-w)}\underline{y}B,
\]
which we consider as an element of $\cH(\SGl)(R)$. Since $\deg(\cH(\SGl))=w$, the lemma follows from \cite[Lemma 2.4]{Ha_cansub}.
\end{proof}

\begin{lem}\label{canlow1}
If $w<(p-1)/p^n$, then we have $\cC_1=\cG[p]_{i_m}=\cG[p]_{i'_m}$ for any integer $m$ satisfying $1\leq m\leq n$.
\end{lem}
\begin{proof}
By \cite[Theorem 1.1 (c)]{Ha_cansubZ}, the equality $\cC_1=\cG[p]_{i_1}$ holds.  From the inequality
\[
i'_n<i_n \leq i'_{n-1}<\cdots <i_2\leq i'_1 <i_1,
\]
we have the inclusions
\[
\cC_1\subseteq \cG[p]_{i'_1}\subseteq \cG[p]_{i_2}\subseteq\cdots \subseteq \cG[p]_{i_n}\subseteq \cG[p]_{i'_n}.
\]
Let us show the reverse inclusion. Let $\SGn$ be the quotient of $\SGm/p\SGm$ in the category $\ModSGf$ corresponding to the closed subgroup scheme $\cC_1\subseteq \cG$. By Theorem \ref{ramcorr2}, it is enough to show the inclusion
\[
\Hom_{\SG,\varphi}(\SGm/p\SGm, m_R^{\geq i'_n}) \subseteq \Hom_{\SG,\varphi}(\SGn,R).
\] 
Consider a $\varphi$-compatible homomorphism of $\SG$-modules $\SGm/p\SGm\to R$ defined by
\[
(\bar{e}_1,\ldots,\bar{e}_h)\mapsto (\underline{x},\underline{y})=\up^{i'_n}(\underline{a},\underline{b})
\]
with an $(h-d)$-tuple $\underline{a}$ and a $d$-tuple $\underline{b}$ in $R$. Then we have the equality
\[
\up^{p i'_n}(\underline{a}^p,\underline{b}^p)=\up^{i'_n}(\underline{a},\underline{b})\begin{pmatrix}I_{h-d} & 0 \\ 0 & u^e I_d\end{pmatrix}\begin{pmatrix}P_1 & P_2 \\ P_3 & P_4\end{pmatrix},
\]
where $\underline{a}^p=(a_1^p,\ldots,a_{h-d}^p)$ and similarly for $\underline{b}^p$.
Multiplying $\begin{pmatrix}Q_1 & Q_2 \\ Q_3 & Q_4\end{pmatrix}=\begin{pmatrix}P_1 & P_2 \\ P_3 & P_4\end{pmatrix}^{-1}$, we obtain the equality
\[
(\underline{a}, u^e \underline{b})=\up^{1/p^n}(\underline{a}^p,\underline{b}^p)\begin{pmatrix}Q_1 & Q_2 \\ Q_3 & Q_4\end{pmatrix}
\]
and we can write $\underline{a}=\up^{1/p^n}\underline{a}'$. The $(h-d)$-tuple $\underline{a}'$ satisfies the equality
\[
\underline{a}'=\up^{1/p^{n-1}-w}(\underline{a}')^p\hat{P}_1-\up^{(p^n-1)/p^n-w}\underline{b}P_3\hat{P}_1.
\]
Hence $v_R(\underline{a}')\geq \min\{1/p^{n-1}, (p^n-1)/p^n\}-w$ and 
\[
v_R(\underline{x})\geq \min\{1/(p^{n-2}(p-1))-w, 1+1/(p^n(p-1))-w\}>w/(p-1).
\]
Since $1-w>w/(p-1)$, we obtain the inequality
\[
v_R(\underline{x}+u^{e(1-w)} \underline{y}B) >w/(p-1).
\]
Then Lemma \ref{canHdg1} implies the reverse inclusion and the lemma follows.
\end{proof}

To show Theorem \ref{canlow}, we proceed by induction on $n$. The case of $n=1$ follows from Lemma \ref{canlow1}. Put $n\geq 2$ and suppose that the theorem holds for any truncated Barsotti-Tate groups of level $n-1$ over $\okey$. Consider a truncated Barsotti-Tate group $\cG$ of level $n$ over $\okey$ with Hodge height $w<(p-1)/p^n$ as in Theorem \ref{canlow}. In particular, we have the equalities $\cC_{n-1}=\cG[p^{n-1}]_{i_{n-1}}=\cG[p^{n-1}]_{i'_{n-1}}$ and thus the inclusions $\cC_{n-1}\subseteq \cG_{i_n}\subseteq \cG_{i'_n}$ also hold.

\begin{lem}\label{plow}
For any positive rational number $i$ satisfying $i \leq 1/(p-1)$, the multiplication by $p$ induces the map $\cG_i(\okbar) \to \cG[p^{n-1}]_{pi}(\okbar)$.
\end{lem}
\begin{proof}
By Lemma \ref{lowramdeg} (\ref{lowramdeg-deg}), we may assume that $\cG$ is connected. By \cite[Th\'{e}or\`{e}me 4.4 (e)]{Il}, there exists a $p$-divisible formal Lie group $\Gamma$ over $\okey$ such that $\cG$ is isomorphic to $\Gamma[p^n]$. By \cite[Lemma 11.3]{Ra}, we can choose formal parameters $X_1,\ldots,X_d$ of the formal Lie group $\Gamma$ such that the multiplication by $p$ of $\Gamma$ is written as
\[
[p](\mathbb{X})\equiv p\mathbb{X}+(X_1^p,\ldots, X_d^p)U+pf(\mathbb{X})\mod \deg p^2,
\]
where $\mathbb{X}=(X_1,\ldots,X_d)$, $f(\mathbb{X})=(f_1(\mathbb{X}),\ldots,f_d(\mathbb{X}))$ such that every $f_l$ contains no monomial of degree less than $p$ and $U \in M_d(\okey)$. Let $\underline{x}=(x_1,\ldots,x_d)$ be a $d$-tuple in $\okbar$ satisfying $[p^n](\underline{x})=0$ and $v_p(\underline{x})\geq i$. By assumption, we have the inequalities $1+v_p(\underline{x})\geq p i$ and $p v_p(\underline{x})\geq p i$. Hence $v_p([p](\underline{x}))\geq pi$ and the lemma follows.
\end{proof}

\begin{lem}\label{inclu}
We have the inclusion $\cG_{i'_n}\subseteq \cC_n$.
\end{lem}
\begin{proof}
By Lemma \ref{canlow1} and Lemma \ref{plow}, the multiplication by $p^{n-1}$ induces a homomorphism $\cG_{i'_n}(\okbar) \to \cG[p]_{i'_1}(\okbar)=\cC_1(\okbar)$. Hence we have the inclusion $\cG_{i'_n}\subseteq p^{-(n-1)}\cC_1$. Consider the natural map $\cG\to \cG/\cC_1$. By \cite[Theorem 1.1]{Ha_cansubZ}, the subgroup scheme $\cC_1\times \sS_{1-w}$ coincides with the kernel of the Frobenius of $\cG\times \sS_{1-w}$. Put $\bar{\cG}=\cG\times \sS_{1-w}$ and similarly for $\overline{\cG/\cC_1}$. Note the inequality $pi'_{n}=i'_{n-1}<1-w$. Then we have a commutative diagram
\[
\xymatrix{
\cG(\okbar)\ar[r]\ar[d] & (\cG/\cC_1)(\okbar)\ar[d] & \\
\bar{\cG}(\cO_{\Kbar,1-w}) \ar[r] & \overline{\cG/\cC_1}(\cO_{\Kbar,1-w}) \ar@{^{(}->}[r]\ar[d] &\bar{\cG}^{(p)}(\cO_{\Kbar,1-w})\ar[d]\\
 & \overline{\cG/\cC_1}(\cO_{\Kbar,p i'_n}) \ar@{^{(}->}[r] &\bar{\cG}^{(p)}(\cO_{\Kbar,p i'_n}),
}
\]
where the composite of the middle row is the Frobenius map and the right horizontal arrows are injections. From this diagram, we see that the map $\cG\to \cG/\cC_1$ induces a map 
\[
\cG_{i'_n}(\okbar)\to (\cG/\cC_1)_{i'_{n-1}}(\okbar).
\]
This implies the inclusion $\cG_{i'_n}/\cC_1\subseteq (p^{-(n-1)}\cC_1/\cC_1)_{i'_{n-1}}$. Note that the group scheme $p^{-(n-1)}\cC_1/\cC_1$ is a truncated Barsotti-Tate group of level $n-1$, height $h$ and dimension $d$ with Hodge height $pw$ and that the subgroup scheme $\cC_n/\cC_1$ is its level $n-1$ canonical subgroup (see the proof of \cite[Theorem 1.1]{Ha_cansub} and \cite[Theorem 1.1]{Ha_cansubZ}). From the induction hypothesis, we see that the equality
\[
(p^{-(n-1)}\cC_1/\cC_1)_{i'_{n-1}}=\cC_n/\cC_1
\]
holds. This implies the inclusion $\cG_{i'_n}\subseteq \cC_n$ and the lemma follows.
\end{proof}

\begin{prop}\label{surj}
The image of the map $\cG_{i_n}(\okbar) \to \cG[p^{n-1}]_{p i_{n}}(\okbar)$ induced by the multiplication by $p$ contains the subgroup $\cG[p^{n-1}]_{i_{n-1}}(\okbar)$.
\end{prop}
\begin{proof}
By Theorem \ref{lowramcorrn} and Lemma \ref{plow}, we have a commutative diagram
\[
\xymatrix{
\cG_{i_n}(\okbar) \ar[r]_-{\sim}\ar[d]_{\times p} & \Hom_{\SG,\varphi}(\SGm, I_{n,i_n}) \ar[d]_{\prjt} \\
\cG[p^{n-1}]_{p i_n}(\okbar)\ar[r]_-{\sim} & \Hom_{\SG,\varphi}(\SGm, I_{n-1, p i_{n}})\\
\cG[p^{n-1}]_{i_{n-1}}(\okbar)\ar@{^{(}->}[u] \ar[r]_-{\sim}& \Hom_{\SG,\varphi}(\SGm, I_{n-1, i_{n-1}}),\ar@{^{(}->}[u]
}
\]
where the horizontal arrows are isomorphisms and the map $\prjt$ is induced by the natural projection $W_n(R)\to W_{n-1}(R)$. It suffices to show that the image of the map $\prjt$ contains the subgroup $\Hom_{\SG,\varphi}(\SGm, I_{n-1, i_{n-1}})$. 

Let $e_1,\ldots,e_h$ be a basis of the $\SG_n$-module $\SGm$ lifting $\bar{e}_1,\ldots,\bar{e}_h$ and $e'_1,\ldots,e'_{h-d}$ be lifts of $\bar{e}'_1,\ldots,\bar{e}'_{h-d}$ in $\SGm$, respectively. Then $e'_1,\ldots,e'_{h-d},e_{h-d+1},\ldots, e_h$ also form a basis of the $\SG_n$-module $\SGm$. Take a $\varphi$-compatible homomorphism of $\SG$-modules $\SGm \to I_{n-1, i_{n-1}}$ defined by
\[
(e'_1,\ldots,e'_{h-d},e_{h-d+1},\ldots,e_h)\mapsto (\underline{x},\underline{y}),
\]
where $\underline{x}=(x_1,\ldots, x_{h-d})$ and $\underline{y}$ are an $(h-d)$-tuple and a $d$-tuple in the ideal $I_{n-1, i_{n-1}}$, respectively. Put $\hat{x}_l=(x_l,0)\in W_n(R)$, $\hat{\underline{x}}=(\hat{x}_1,\ldots,\hat{x}_{h-d})$ and similarly for $\hat{\underline{y}}$. Let $A$ be the matrix in $M_h(\SG_n)$ satisfying
\[
\varphi(e'_1,\ldots,e'_{h-d},e_{h-d+1},\ldots,e_h)=(e'_1,\ldots,e'_{h-d},e_{h-d+1},\ldots,e_h)A.
\]
Define an $(h-d)$-tuple $\underline{\xi}=(\xi_1,\ldots,\xi_{h-d})$ and a $d$-tuple $\underline{\eta}$ in $R$ by
\[
p^{n-1}([\underline{\xi}], [\underline{\eta}])=\varphi(\hat{\underline{x}}, \hat{\underline{y}})-(\hat{\underline{x}}, \hat{\underline{y}})A,
\]
where we put $[\underline{\xi}]=([\xi_1],\ldots,[\xi_{h-d}])$ and similarly for $[\underline{\eta}]$.
By Proposition \ref{Y}, the elements $\hat{x}$ and $\hat{y}$ are divisible by $[\up^{i_{n-1}}]$ and thus we can write
\[
(\underline{\xi}, \underline{\eta})=\up^{i_{n-1}} (\underline{\xi}', \underline{\eta}').
\]
Since $i_{n-1}=p i_n+w\geq p i_n$, Lemma \ref{ipi} implies that, for any $h$-tuple $\underline{z}$ in $R$, the element $(\hat{\underline{x}},\hat{\underline{y}})+p^{n-1}[\up^{i_n}\underline{z}]$ is contained in the ideal $I_{n,i_n}$. It is enough to show that there exists an $h$-tuple $\underline{z}$ in $R$ satisfying
\[
\varphi((\hat{\underline{x}},\hat{\underline{y}})+p^{n-1}[\up^{i_n}\underline{z}])=((\hat{\underline{x}},\hat{\underline{y}})+p^{n-1}[\up^{i_n}\underline{z}])A.
\]
Put $\underline{z}=(\underline{\zeta}, \underline{\omega})$ with an $(h-d)$-tuple $\underline{\zeta}$ and a $d$-tuple $\underline{\omega}$. Then this is equivalent to the equation
\[
(\underline{\xi}, \underline{\eta})+ \up^{p i_n}(\underline{\zeta}^p, \underline{\omega}^p)=\up^{i_n}(\underline{\zeta}, \underline{\omega})\begin{pmatrix}D & P_2\\ 0 & u^{e(1-w)} P'_4 \end{pmatrix}.
\]
We claim that the equation for the first entry
\[
\underline{\xi}+\up^{p i_n}\underline{\zeta}^p=\up^{i_n}\underline{\zeta}D
\]
has a solution $\underline{\zeta}=\up^{(p-1) i_n}\underline{\zeta}'$ with an $(h-d)$-tuple $\underline{\zeta}'$ in $R$. Indeed, let $\hat{D}\in M_{h-d}(k[[u]])$ be the matrix satisfying $D\hat{D}=u^{ew}I_{h-d}$. Then this is equivalent to the equation
\[
\underline{\zeta}'=\underline{\xi}'\hat{D}+\up^{p(p-1) i_n-w}(\underline{\zeta}')^p\hat{D}.
\]
Since $p(p-1)i_n>w$, we can find a solution $\underline{\zeta}'$ of the equation by recursion. For the second entry, we have the equation
\[
\up^{p i_n+w}\underline{\eta}'+\up^{p i_n}\underline{\omega}^p=\up^{i_n}(\underline{\zeta}P_2+\up^{1-w} \underline{\omega} P'_4).
\]
This is equivalent to the equation
\[
\underline{\omega}^p=\up^{1-w-(p-1)i_n}\underline{\omega} P'_4+\underline{\zeta}' P_2-\up^w \underline{\eta}'.
\]
Note the inequality $1-w \geq (p-1)i_n$. Write this equation as
\[
(\omega^p_1,\ldots,\omega^p_d)+(\omega_1,\ldots,\omega_d)C+(c'_1,\ldots,c'_d)=0
\]
with some $C=(c_{i,j})\in M_d(R)$ and $c'_i\in R$. Then the $R$-algebra
\[
R[\omega_1,\ldots,\omega_d]/(\omega^p_1+\sum_{j=1}^d c_{j,1}\omega_j+c'_1,\ldots,\omega^p_d+\sum_{j=1}^d c_{j,d}\omega_j+c'_d)
\]
is free of rank $p^d$ over $R$. Since $\Frac(R)$ is algebraically closed and $R$ is integrally closed, this $R$-algebra admits at least one $R$-valued point. Hence we can find at least one solution $\underline{\omega}$ of the equation. This concludes the proof of the proposition.
\end{proof}

Consider the exact sequence
\[
0 \to \cG[p]_{i_n}(\okbar) \to \cG_{i_n}(\okbar)  \overset{\times p}{\to} \cG[p^{n-1}]_{p i_n}(\okbar).
\]
Proposition \ref{surj} implies that the image of the rightmost arrow contains the subgroup 
\[
\cG[p^{n-1}]_{i_{n-1}}(\okbar)\subseteq \cG[p^{n-1}]_{p i_n}(\okbar),
\]
which coincides with $\cC_{n-1}(\okbar)$ by induction hypothesis and thus is of order $p^{(n-1)d}$. By Lemma \ref{canlow1}, the subgroup $\cG[p]_{i_n}(\okbar)$ also coincides with $\cC_1(\okbar)$ and this is of order $p^d$. Hence the group $\cG_{i_n}(\okbar)$ is of order no less than $p^{nd}$. Since Lemma \ref{inclu} implies the inclusions
\[
\cG_{i_n}(\okbar) \subseteq \cG_{i'_n}(\okbar)\subseteq \cC_n(\okbar),
\]
Theorem \ref{canlow} follows by comparing orders. \qed

To prove Theorem \ref{familycan}, we need the following lemma, which is a ``lower'' variant of \cite[Lemma 4.5]{Ha_cansub}.

\begin{lem}\label{familylow}
Let $K/\bQ_p$ be an extension of complete discrete valuation fields and $i$ be a positive rational number. Let $\frX$ be an admissible formal scheme over $\Spf(\okey)$ and $X$ be its Raynaud generic fiber. Let $\frG$ be a finite locally free formal group scheme over $\frX$ with Raynaud generic fiber $G$.
Then there exists an admissible open subgroup $G_i$ of $G$ over $X$ such that the open immersion $G_i\to G$ is quasi-compact and that for any finite extension $L/K$ and $x\in X(L)$, the fiber $(G_i)_x$ coincides with the lower ramification subgroup $(\frG_{x})_i\times \Spec(L)$ of the finite flat group scheme $\frG_{x}=\frG\times_{\frX,x}\Spf(\oel)$ over $\oel$. 
\end{lem}
\begin{proof}
Let $\sI$ be the augmentation ideal sheaf of the formal group scheme $\frG$. Write $i=m/n$ with positive integers $m,n$ and put $\sJ=p^m\cO_{\frG}+\sI^n$. Let $\frB$ be the admissible blow-up of $\frG$ along the ideal $\sJ$ and $\frG_{m,n}$ be the formal open subscheme of $\frB$ where $p^m$ generates the ideal $\sJ\cO_\frB$. Since the Raynaud generic fiber of $\frG_{m,n}$ is the admissible open subset of $G$ whose set of $\Kbar$-valued points is given by
\[
\{x\in G(\Kbar)| v_p(\sI(x))\geq i\},
\]
it is independent of the choice of $m,n$ and we write it as $G_i$. Using the universality of dilatations as in the proof of \cite[Proposition 8.2.2]{AM}, we can show that $G_i$ is an admissible open subgroup of the rigid-analytic group $G$. For any affinoid open subset $U=\Spv(A)$ of $G$, put $I=\Gamma(U,\sI)$. Then the intersection $U\cap G_i$ is the affinoid $\Spv(A\langle I^n/p^m\rangle)$ and thus the open immersion $G_i\to G$ is quasi-compact. This concludes the proof of the lemma.
\end{proof}

\begin{proof}[Proof of Theorem \ref{familycan}]
Set $C_n$ to be the admissible open subgroup $G_{i'_n}$ of $G$ as in Lemma \ref{familylow} with $i'_n=1/(p^n(p-1))$. Then, by this lemma and Theorem \ref{canlow}, each fiber $(C_n)_x$ coincides with the generic fiber of the level $n$ canonical subgroup of $\frG_x$ and its group of $\Kbar$-valued points is isomorphic to the group $(\bZ/p^n\bZ)^d$. Moreover, $C_n$ is etale, quasi-compact and separated over $X(r_n)$. Thus \cite[Theorem A.1.2]{Co} implies that $C_n$ is finite over $X(r_n)$ and the theorem follows by a similar argument to the proof of  \cite[Corollary 1.2]{Ha_cansub}.
\end{proof}




\begin{thebibliography}{9a}




\bibitem{AM}
A. Abbes and A. Mokrane: \emph{Sous-groupes canoniques et cycles \'evanescents $p$-adiques pour les vari\'et\'es ab\'eliennes}, Publ. Math. Inst. Hautes Etudes Sci. {\bf 99} (2004), 117--162. 


\bibitem{BBM}
P. Berthelot, L. Breen and W. Messing: \emph{Th\'eorie de Dieudonn\'e cristalline II}, Lecture Notes in Mathematics, {\bf 930}. Springer-Verlag, Berlin, 1982. x+261 pp.


\bibitem{Br}
C. Breuil: \emph{Groupes $p$-divisibles, groupes finis et modules filtr\'{e}s}, Ann. of Math. (2) {\bf 152} (2000), no. 2, 489--549. 

\bibitem{Br_AZ}
C. Breuil: \emph{Integral $p$-adic Hodge theory}, Algebraic geometry 2000, Azumino (Hotaka), 51--80, Adv. Stud. Pure Math., {\bf 36}, Math. Soc. Japan, Tokyo, 2002. 


\bibitem{Co}
B. Conrad: \emph{Modular curves and rigid-analytic spaces}, Pure Appl. Math. Q. {\bf 2} (2006), no. 1, 29--110. 


\bibitem{deJM}
A. de Jong and W. Messing: \emph{Crystalline Dieudonn\'{e} theory over excellent schemes}, Bull. Soc. Math. France {\bf 127} (1999), no. 2, 333--348.

\bibitem{Fal}
G. Faltings: \emph{Integral crystalline cohomology over very ramified valuation rings}, J. Amer. Math. Soc. {\bf 12} (1999), 117--144. 



\bibitem{Fa}
L. Fargues: \emph{La filtration canonique des points de torsion des groupes p-divisibles (avec la collaboration de Yichao Tian)}, Ann. Sci. Ecole Norm. Sup. (4) {\bf 44} (2011), no. 6, 905--961. 





\bibitem{SGA3-7A}
P. Gabriel: \emph{\'Etude infinit\'esimale des sch\'emas en groupes et groupes formels}, Sch\'emas en Groupes (S\'em. G\'eom\'etrie Alg\'ebrique, Inst. Hautes Etudes Sci., 1963/64), Fasc. 2b, Expos\'e 7a, pp. 1-65+4, Inst. Hautes Etudes Sci., Paris.




\bibitem{Ha_ramcorr}
S. Hattori: \emph{Ramification correspondence of finite flat group schemes over equal and mixed characteristic local fields}, J. of Number Theory {\bf 132} (2012), no. 10, 2084--2102.

\bibitem{Ha_cansub}
S. Hattori: \emph{Canonical subgroups via Breuil-Kisin modules}, to appear in Math. Z.

\bibitem{Ha_cansubZ}
S. Hattori: \emph{Canonical subgroups via Breuil-Kisin modules for $p=2$}, available at \verb|http://www2.math.kyushu-u.ac.jp/~shin-h/|


\bibitem{Il}
L. Illusie: \emph{D\'{e}formations de groupes de Barsotti-Tate (d'apr\`{e}s A. Grothendieck)}, Seminar on arithmetic bundles: the Mordell conjecture (Paris, 1983/84), Asterisque No. {\bf 127} (1985), 151--198. 


\bibitem{Kim_2}
W. Kim: \emph{The classification of $p$-divisible groups over $2$-adic discrete valuation rings}, Math. Res. Lett. {\bf 19} (2012), no. 1, 121--141.

\bibitem{Ki_Fcrys}
M. Kisin: \emph{Crystalline representations and $F$-crystals}, Algebraic geometry and number theory, 459--496, Progr. Math. {\bf 253}, Birkhauser Boston, Boston, MA, 2006.


\bibitem{Ki_2}
M. Kisin: \emph{Modularity of $2$-adic Barsotti-Tate representations}, Invent. Math. {\bf 178} (2009), no. 3, 587--634. 


\bibitem{Lau_2}
E. Lau: \emph{A relation between Dieudonn\'{e} displays and crystalline Dieudonn\'{e} theory}, arXiv:1006.2720v2 (2010). 


\bibitem{Li_BC}
T. Liu: \emph{On lattices in semi-stable representations: a proof of a conjecture of Breuil}, Compos. Math. {\bf 144} (2008), no. 1, 61--88. 

\bibitem{Li_2}
T. Liu: \emph{The correspondence between Barsotti-Tate groups and Kisin modules when $p=2$}, available at \verb|http://www.math.purdue.edu/~tongliu/research.html|.

\bibitem{Ra}
J. Rabinoff: \emph{Higher-level canonical subgroups for $p$-divisible groups}, J. Inst. Math. Jussieu {\bf 11} (2012), no. 2, 363--419.



\bibitem{Ti}
Y. Tian: \emph{Canonical subgroups of Barsotti-Tate groups}, Ann. of Math. (2) {\bf 172} (2010), no. 2, 955--988.

\bibitem{Ti_HN}
Y. Tian: \emph{An upper bound on the Abbes-Saito filtration for finite flat group schemes and applications}, Algebra \& Number Theory {\bf 6} (2012), no. 2, 231--242.




\end{thebibliography}
\end{document}